\documentclass{amsart}
\usepackage{amssymb,latexsym}
\theoremstyle{plain}
\newtheorem{theorem}{Theorem}

\newtheorem{lemma}{Lemma}
\theoremstyle{definition}

\newtheorem{remark}{Remark}

\DeclareMathOperator{\Res}{Res}
\DeclareMathOperator{\red}{red}
\DeclareMathOperator{\Sing}{Sing}

\date{}

\begin{document}

\title[lines]
{Postulation of disjoint unions of lines and a multiple point, II}
\author{E. Ballico}
\address{Dept. of Mathematics\\
 University of Trento\\
38123 Povo (TN), Italy}
\email{ballico@science.unitn.it}
\thanks{The author was partially supported by MIUR and GNSAGA of INdAM (Italy).}

\begin{abstract}
We study the postulation of a general union $X\subset \mathbb {P}^3$ of one m-point $mP$
and $t$ disjoint lines. We prove that it has the expected Hilbert function,
proving a conjecture by E. Carlini, M. V. Catalisano
and A. V. Geramita.
\end{abstract}

\maketitle

\section{Introduction}\label{S1}

A scheme $X\subset \mathbb {P}^r$ is said to have {\it maximal rank} if for all integers $t>0$ the restriction map $H^0(\mathcal {O}_{\mathbb {P}^r}(t)) \to H^0(X,\mathcal {O}_X(t))$
is either injective or surjective, i.e. if either $h^0(\mathcal {I}_X(t)) =0$ or $h^1(\mathcal {I}_X(t)) =0$, i.e. if $X$ imposes the ``~expected~'' number of conditions to the vector space of all homogeneous degree $t$ polynomials in $r+1$ variables. R. Hartshorne and A. Hirschowitz proved that for all integers $t>0$ and $r\ge 3$ a general union $X \subset \mathbb {P}^r$
of $t$ general lines has maximal rank. E. Carlini, M. V. Catalisano and A. V. Geramita considered several cases in which we allow unions of linear spaces with certain multiplicities \cite{ccg2}, \cite{ccg3}, \cite{ccg1}). We recall that for each $P\in \mathbb {P}^r$ the m-point $mP$ of $\mathbb {P}^r$ is the closed subscheme
of $\mathbb {P}^r$ with $(\mathcal {I}_P)^m$ as its ideal sheaf. E. Carlini, M. V. Catalisano and A. V. Geramita  proved that for all $r\ge 4$, $m>0$ and $d>0$ a general union of an m-point and $d$ disjoint lines
has maximal rank (\cite{ccg1}). In the case $r=3$ they proved that there are some exceptional
cases (the one with $2 \le d \le m$ and $t=m$); in \cite{ccg1} the failure of maximal rank for these cases  is exactly described, i.e.
all positive integers $h^0(\mathcal {I}_X(t))$ and $h^1(\mathcal {I}_X(t)) $ are computed (\cite[Theorem 4.2, part (ii)]{ccg1}). They conjectured in \cite{ccg1} that these
are the only exceptional cases and proved the conjecture
in some cases (e.g. if $m=2$ by \cite[Theorem 4.2, part (i)(e)]{ccg1}). In \cite{ab} their conjecture was proved when $m=3$ and an asymptotic result was proved
for arbitrary $m$ (\cite[Propositions 1 and 2]{ab}). In this paper we prove their conjecture in the case $m = 3$, i.e. we prove the following result.

\begin{theorem}\label{i1}
Fix integers $m\ge 2$, $t>0$ and $d >0$. If $2 \le d \le m$, then assume $t \ge m+1$. Let $Y\subset \mathbb {P}^3$ be a general union of $d$ lines.
Then either $h^1(\mathcal {I}_{mP\cup Y}(t)) =0$ or $h^0(\mathcal {I}_{mP\cup Y}(t)) =0$.
\end{theorem}

A crucial step of the proof is contained in \cite[Theorem 4.2, part (i)(c)]{ccg1}: the proof of the case $d= m+2$ and $t=m+1$. Let $Y\subset \mathbb {P}^3$ be a general union of $m+2$
lines. They proved that $h^i(\mathcal {I}_{mP\cup Y}(m+1))=0$, $i=0,1$. After \cite{hi} and \cite{hh} it is well-known that if certain crucial curves or unions of curves
and points, say $X_1$ and $X_2$, have
$h^i(\mathcal {I}_{X_1}(t_0))=0$, $i=0,1$, and $h^i(\mathcal {I}_{X_2}(t_0+1)) =0$, $i=0,1$, then it should be easy to control the postulation of all curves of degree $\ge \deg (X_2)$ with respect to
all forms of degree $\ge t_0+2$. In our case by \cite[Theorem 4.2, part (i)(c)]{ccg1} we may take $X_1 =mP\cup Y$ with $\deg (Y)=m+2$. The key part of the proof is the construction
of
a good $X_2$ for $t_0+1 =m+2$ and then to control the cases $t=t_0+3$ and $t =t_0+4$.

We work over an algebraically closed field $\mathbb {K}$. As far as we understand none of our quotations of \cite{ccg1} require the characteristic zero assumption
made in \cite{ccg1}.

\section{Preliminaries}\label{Sp}
For any integer $d> 0$ let $L(d)$ be the set of all unions $Y\subset \mathbb {P}^3$ of $d$ disjoint lines. For any $P\in \mathbb {P}^3$ set $L(P,d):= \{Y\in L(d): P\notin Y\}$.
If $P$ is a smooth point of a scheme $T$ let $\{mP,T\}$ be the closed subscheme of $T$ with $(\mathcal {I}_{P,T})^m$ as its ideal sheaf. We write $mP$ instead of $\{mP,\mathbb {P}^3\}$. For any positive-dimensional $A\subseteq \mathbb {P}^3$ and any smooth point $O$ of $A$ a {\it tangent vector} of $A$ with $O$ as its support is a degree $2$ connected zero-dimensional scheme $v\subset A$
such that $\deg (v) =2$ and $v_{\red }=\{O\}$.

Let $F \subset \mathbb {P}^3$ be any surface. Set $t:= \deg (F)$. For each closed subscheme $Z\subset \mathbb {P}^3$ let $\Res_F(Z)$ denote the residual scheme
of $Z$ with respect to $F$, i.e. the closed subscheme of $\mathbb {P}^3$ with $\mathcal {I}_Z:\mathcal {I}_F$ as its ideal sheaf. If $Z$ is reduced, then $\Res_F(Z)$
is the union of the irreducible components of $Z$ not contained in $F$. Now assume $Z = mP$ for some $m>0$ and some $P\in \mathbb {P}^3$. If $P\notin F$,
then $\Res_F(mP) = mP$. If $P$ is a smooth point of $F$, then $\Res_F(mP) = (m-1)P$ (with the convention $0P = \emptyset$). For any integer $x\ge t$ we have an exact sequence
$$0\to  \mathcal {I}_{\Res_F(Z)}(x-t) \to \mathcal {I}_Z(x) \to \mathcal {I}_{Z\cap F,F}(x) \to 0$$Hence
\begin{itemize}
\item $h^0(\mathcal {I}_Z(x)) \le h^0(\mathcal {I}_{\Res_F(Z)}(x-t)) + h^0(F,\mathcal {I}_{Z\cap F}(x))$;
\item $h^1(\mathcal {I}_Z(x)) \le h^1(\mathcal {I}_{\Res_F(Z)}(x-t)) + h^1(F,\mathcal {I}_{Z\cap F}(x))$.
\end{itemize}
As in \cite{ccg2}, \cite[Lemma 3.3]{ccg3} and \cite{ccg1} we will call ``~the Castelnuovo's inequality~'' any of these two inequalities. If $F$ is either
a plane or a smooth quadric, $D$ is an effective divisor of $F$ and $Z\subset F$ is a closed subscheme of $F$, $\Res _D(Z)$ is the closed
subscheme of $F$ with $\mathcal {I}_{Z,F}:\mathcal {I}_{D,F}$ as its ideal sheaf (of course, $\mathcal {I}_{D,F}\cong \mathcal {O}_F(-D)$ as abstract line bundles on $F$).
We also have the corresponding Castelnuovo's exact sequence of $\Res _D$ and the associated Castelnuovo's inequalities.

Set $0P:= \emptyset$. We use the convention that $\binom{t}{3} =0$ if $-2\le t \le 2$ and $\binom{t+2}{2} =0$ if $-1\le t \le 1$. We have $\deg (0P) = 0 = \binom{2}{3}$. For all integers $m\ge 0$ and $k\ge 0$ define the integers $a_{m,k}$ and $b_{m,k}$ by the relations
\begin{equation}\label{eqa1}
\binom{m+1}{3} +(k+1)a_{m,k} +b_{m,k} =\binom{k+3}{3}, 0 \le b_{m,k} \le k
\end{equation}
If $k\ge 2$ from (\ref{eqa1}) for $k$, $k-2$, $k-1$ and $m-1$ we get 
\begin{equation}\label{eqa2}
2a_{m,k-2} +(k+1)(a_{m,k}-a_{m,k-2}) + b_{m,k}-b_{m,k+1} =(k+1)^2
\end{equation}
Taking the difference of (\ref{eqa2}) with $k=m+2$ and the same equation with $(m',k') = (m-1,m+1)$ and using that
$\binom{m+2}{3} -\binom{m+1}{3} = \binom{m+1}{2}$ and $\binom{m+4}{2} -\binom{m+1}{2} = 3m+6$, we get
\begin{equation}\label{eqa3}
a_{m-1,m+1} + (m+3)(a_{m,m+2} -a_{m-1,m+2}) +b_{m,m+2} -b_{m-1,m+1} = 3m+6
\end{equation}
for all $m>0$. Taking $k=m+2$ in (\ref{eqa1}) we get
\begin{equation}\label{eqa4}
(m+3)a_{m,m+2} +b_{m,m+2} = (3m^2+15m+30)/2
\end{equation}

\begin{remark}\label{a1}
We have $b_{m,m+1} =0$ and $a_{m,m+1} = m+2$ for all $m$. From (\ref{eqa4}) we get that if $m$ is even, then $a_{m,m+2} = 3m/2 +3$ and $b_{m,m+2} =1$, while if $m$ is
odd, then $a_{m,m+2} = 3m/2 +5/2$ and $b_{m,m+2} = m/2 + 5/2$. Hence for all $m\ge 3$ we have $ a_{m-1,m+1}  > m$, $a_{m,m+2}=a_{m-1,m+1}+2$
if $m$ is even and $a_{m,m+2} = a_{m-1,m+1} +1$ if $m$ is odd.
We have $\binom{m+6}{3} -\binom{m+2}{3} =2m^2+12m+10$ and hence $a_{m,m+3} =2m+4$ and $b_{m,m+3} =4$ for all $m\ge 1$, $a_{0,3} = 6$, $b_{0,3} =2$.
We have $\binom{m+7}{3} -\binom{m+2}{3} = (5m^2+35m+70)/2$. If $m$ is even and $m\ge 6$, then $a_{m,m+4} = 5m/2+5$ and $b_{m,m+4} =10$. If $m\in \{2,4\}$,
then $a_{m,m+4} =5m/2+6$ and $b_{m,m+4} =5-m$.
If $m$ is odd and $m\ge 17$, then $a_{m,m+4} =5m/2+9/2$ and $b_{m,m+4} = (m+25)/2$. If $m\in \{3,5,7,9,11,13,15\}$, then $a_{m,m+4} = 5m/2+11/2$ and $b_{m,m+4}
= (15-m)/2$. 
\end{remark}

For all positive integers $m, d$ the {\it critical value} of the pair $(m,d)$ is the minimal integer $k\ge m$ such that $\binom{m+2}{3} +(k+1)d\le \binom{k+3}{3}$.
Let $W\subset \mathbb {P}^3$ be a union of $d$ disjoint lines with $P\notin W$. The scheme $mP\cup W$ has maximal rank if and only if $h^0(\mathcal {I}_{mP\cup W}(k-1))=0$
and $h^1(\mathcal {I}_{mP\cup W}(k)) =0$, where $k$ is the critical value of $(m,d)$. Using (\ref{eqa2}) it is easy to check that for a fixed integer $m>0$ the sequence $a_{m,k}$ is strictly increasing for
all $k\ge m-1$ (we have $a_{m,m-1}=0$). The integer $k$ is the critical value of the pair $(m,d)$ if and only if $a_{m,k-1}< d \le a_{m,k}$.

\section{Assertions $B(m)$, $R(m)$ and $H_{m,k}$}\label{S3}

For every odd positive integer $m$ we define Assertion $B(m)$ in the following way.

\quad {\bf Assertion} $B(m)$, $m \ge 1$, $m$ odd: There is a 7-ple $(Y,L,R,S,O, H,v)$ with the following properties:
\begin{enumerate}
\item $H$ is a plane containing $P$,  $L$ and $R$ are lines of $H$, $L\ne R$,  $P\notin L\cup R$, and $\{O\}:= L\cap R$;
\item $Y$ is a union of $a_{m,m+2}$ disjoint lines, $P\notin Y$ and $Y\cap H$ is finite;
\item $S\subset H\cap Y$, $\sharp (S)=b_{m,m+2}-2$;
\item $v$ is a disjoint union of $b_{m,m+2}-2$ tangent vectors of $\mathbb {P}^3$, each of them with a point of $S$ as its support;
\item $\sharp (S\cap L) = \lceil (m+3)/4 \rceil$, $\sharp (S\cap R) = \lfloor (m+3)/4\rfloor$
and $L\cap R \cap S=\emptyset$;
\item $h^1(\mathcal {I}_{mP\cup Y\cup v\cup \{O\}}(m+2)) =0$.
\end{enumerate}

Take $(Y,L,R,S,O,H,v)$ satisfying the second, third and fourth of the conditions of $B(m)$. We have $h^0(\mathcal {O}_{mP\cup Y\cup v\cup \{O\}}(m+2)) =\binom{m+5}{3}-1$
and hence $h^1(\mathcal {I}_{mP\cup Y\cup v\cup \{O\}}(m+2)) =h^0(\mathcal {I}_{mP\cup Y\cup v\cup \{O\}}(m+2))-1$.

For every even integer $m\ge 2$ we define Assertion $B(m)$ in the following way.

\quad {\bf Assertion} $B(m)$, $m \ge 2$, $m$ even: There is a quadruple $(Y,L,R,H)$ with the following properties:
\begin{enumerate}
\item $H$ is a plane containing $P$,  $L$ and $R$ are lines of $H$, $L\ne R$, and $P\notin L\cup R$;
\item $Y$ is a union of $a_{m,m+2}$ disjoint lines, $P\notin Y$ and $Y\cap H$ is finite;
\item $\sharp ((Y\cap H)\cap L) = \lceil (m+2)/4 \rceil$, $\sharp ((Y\cap H)\cap R) = \lfloor (m+2)/4 \rfloor$;
\item $h^1(\mathcal {I}_{mP\cup Y}(m+2)) =0$.
\end{enumerate}

The last condition of $B(m)$, $m$ even, is equivalent to $h^0(\mathcal {I}_{mP\cup Y}(m+2)) =1$.

\begin{lemma}\label{a8}
$B(m)$ is true for all $m \ge 2$.
\end{lemma}

\begin{proof}

We first prove $B(2)$. Let $Y\subset \mathbb {P}^3$ be a general union of $6$ lines (hence $P\notin Y$). By \cite[part (i)(e) of Theorem 4.2]{ccg1} we have $h^1(\mathcal {I}_{2P\cup Y}(4)) =0$.
Let $H\subset \mathbb {P}^3$ be a general plane though $P$. Moving $Y$ we see that we may assume that no $3$ of the points of $(Y\cap H)\cup \{P\}$ are collinear.

Now assume $m\ge 3$
and that $B(m-1)$ is true.

\quad (a) In this step we assume that $m$ is odd. Take $(Y,L,R,S,H)$ satisfying $B(m-1)$. We have $h^1(\mathcal {I}_{(m-1)P\cup Y}(m+1))=1$.
Let $D\subset H$ be a general line. Let $v\subset H$ be a union of tangent vectors of $H$ with $S$ as its support, but no tangent vector
being a tangent vector of $L\cup R$. We first check that $h^1(\mathcal {I}_{mP\cup Y\cup D\cup v\cup \{O\}}(m+2))=0$. Since $\Res _H(mP\cup Y\cup D \cup v\cup \{O\}) =(m-1)P\cup Y$
and $h^1(\mathcal {I}_{(m-1)P\cup Y}(m+1)) =0$, it is sufficient to prove that $h^1(H,\mathcal {I}_{((mP\cup Y)\cap H)\cup D\cup v \cup \{O\}}(m+2)) =0$,
i.e. $h^1(H,\mathcal {I}_{((mP\cup Y)\cap H)\cup v \cup \{O\}}(m+1)) =0$. The scheme $((mP\cup Y)\cap H)\cup v \cup \{O\}$ is a general
union of $\{mP,H\}$, the scheme $v\cup \{O\}$ and $a_{m-1,m+1} -(m+1)/2$ general points of $H$. Hence it has degree $\binom{m+1}{2}
+2(m+1)/2 + 3(m-1)/2 +3 -(m+1)/2 +1 = \binom{m+3}{2}$. We deform $D$ in a flat family of lines outside $H$ (we may do it even fixing either the point of $D\cap L$
or the point of $D\cap R$). For general $v$ it is easy to check that $h^1(H,\mathcal {I}_{\{mP,H\}\cup v\cup \{O\}}(m+1)) =0$ (order the points of $S$ and then add
the corresponding connected component $v_i$ of $v$ following the ordering first with the point $P_i$ of $S$ general in a component of $L\cup R$ and then with
$v_i$ general among the tangent vectors of $H$ with $P_i$ as its support; at each point use that $h^0(H,\mathcal {O}_{\{mP,H\}}(m)) =m+1$ and that if $P_i\in L_i$, then
$|\mathcal {I}_{\{mP,H\}\cup \{2P_i,L_i\}}(m+1)| \cong |\mathcal {I}_{\{mP,H\}}(m)|$). Since $Y\cap H \setminus S$ is general in $H$, we get $h^i(H,\mathcal {I}_{((mP\cup Y)\cap H)\cup v \cup \{O\}}(m+1)) =0$, $i=0,1$. 

\quad (b) In this step we assume that $m$ is even. Take $(Y,L,R,S,O, H,v)$ satisfying $B(m-1)$. Let $w\subset \mathbb {P}^3$ be a general tangent vector with $O$ as its support.
The scheme $Y\cup L\cup R \cup w\cup v$ is a flat limit of a family of disjoint unions of $a_{m,m+2}$ lines (i.e. there are a flat family $\{Y_t\}_{t\in \Gamma }$, $\Gamma$ an integral affine curve, $o\in \Gamma$, $Y_o = Y\cup L\cup R \cup w\cup v$) such that
$Y\subset Y_t$, say $Y_t = Y\cup L_t\cup R_t$ for all $t$ with $\{L_t\}$, and $\{R_t\}$ flat families with $L_o=L$ and $L_o=L$,  and either $L_t\cap L \ne \emptyset$ for all
$t$, $R_t\cap R =\emptyset$ for all $t\ne o$ (case $m\equiv 2 \pmod{4}$) or $R_t\cap R \ne \emptyset$ for all $t$ and $L_t\cap L=\emptyset$ for
all $t\ne o$ (case $m\equiv 0 \pmod{4}$). We may take as the new set $S$ the set $S\cup (L_t\cup R_t)\cap (L\cup R)$ for a general $t\in \Gamma$.
By the semicontinuity theorem for cohomology (\cite[III.12.8]{h}) it is sufficient to prove that $h^1(\mathcal {I}_{mP\cup Y\cup  L\cup R\cup v\cup w}(m+2)) =0$.
Since $(mP\cup Y\cup L\cup R \cup v\cup w)\cap H = \{mP,H\}\cup L\cup R \cup ((Y\cap H)\setminus S)$ and $\Res _H(mP\cup Y\cup  L\cup R\cup v\cup w)
=(m-1)P\cup Y\cup v\cup \{O\}$, it is sufficient to prove that $h^1(H,\mathcal {I}_{\{mP,H\}\cup L\cup R \cup ((Y\cap H)\setminus S)}(m+2)) =0$,
i.e. $h^1((H,\mathcal {I}_{\{mP,H\}\cup ((Y\cap H)\setminus S)}(m)) =0$. This is true, because $(Y\cap H)\setminus S$ is general in $H$
and $\sharp ((Y\cap H)\setminus S) = 3m/2+1 -m/2 =m+1 =h^0(H,\mathcal {I}_{\{mP,H\}}(m)$.
\end{proof}

\begin{remark}\label{a8.0}
Fix $(m,d)$ with critical value $m+1$ and degree $d\ge m+1$. Let $W\subset \mathbb {P}^m$ be a general union of $d$ lines. Since
$a_{m,m+1}=m+2$, we have $m+1 \le d \le m+2$. By \cite[part (i)(c) of Theorem 4.2]{ccg1} the scheme $mP\cup W$ has maximal rank.
\end{remark}

\begin{lemma}\label{a9}
Fix an integer $d\le a_{m,m+2}$ and let $X\subset \mathbb {P}^3$ be a general union of $d$ lines. Then $h^1(\mathcal {I}_{mP\cup X}(m+2)) =0$.
\end{lemma}

\begin{proof}
This statement is obvious if $m=1$ by \cite{hh}. Assume $m\ge 2$. It is sufficient to find a disjoint union $W$ of $d$ lines
such that $P\notin W$ and $h^1(\mathcal {I}_{mP\cup W}(m+2))=0$. Take a solution of $B(m)$ and call $Y$ the curve in it. Take as $W$ the union
of $d$ of the lines of $Y$.
\end{proof}

For all odd integer $m\ge 3$ let $R(m)$ denote the following assertion:

\quad {\bf Assertion} $R(m)$, $m$ odd, $r\ge 3$: There exists a quintuple $(Y,S,D,H,v)$ with the following properties:
\begin{enumerate}
\item $Y\subset \mathbb {P}^3$ is a disjoint union of $3m/2+5/2$ lines, $P\notin Y$, $H$ is a plane containing $P$, $D\subset H$
is a smooth conic such that $P\notin D$ and $S:= (Y\cap H)\cap D$ has cardinality $m/2+5/2$;
\item $v \subset \mathbb {P}^3$ is a disjoint union of tangent vectors of $\mathbb {P}^3$ with $v_{\red }=S$; no connected component
of $v$ is contained in $Y$;
\item $h^i(\mathcal {I}_{mP\cup Y\cup v}(m+2)) =0$. 
\end{enumerate}

\begin{lemma}\label{a12.1}
$R(m)$ is true for all odd integers $m\ge 3$. 
\end{lemma}

\begin{proof}
Take  $(Y,L,R,H)$ satisfying $B(m-1)$. We have $\sharp (Y\cap (L\cup R)) = (m+1)/2$, $h^1(\mathcal {I}_{(m-1)P\cup Y}(m+1)) =0$
and $h^0(\mathcal {I}_{(m-1)P\cup Y}(m+1)) =1$. Since $h^0(\mathcal {I}_{(m-2)P\cup Y}(m)) =0$, $P\in H$, and $Y\cap (R\cup L) \ne Y\cap H$,
there is $o\in L\cup R$ not in the base locus of $|\mathcal {I}_{(m-1)P\cup Y}(m+1)|$ and
hence $h^i(\mathcal {I}_{(m-1)P\cup Y\cup \{o\}}(m+1)) =0$, $i=0,1$. We may deform $(Y,L\cup R,o)$ to $(Y',C,o')$, where $C\subset H$ is a smooth conic, $P\notin C$, $o'\in C\setminus C\cap Y$,
$\sharp (Y'\cap C) = (m+1)/2$ and $h^0(\mathcal {I}_{(m-1)P\cup Y\cup \{o'\}}(m+1))=0$. We may take as $o'$ a general point of $C$. Let $o''$ be another general point
of $C$ and call $T$ the line spanned by $o'$ and $o''$ (alternatively, take a general line $T\subset H$ and set $\{o',o''\}:= C\cap T)$.
Let $w\subset H$ be a general union of tangent vectors of $H$, each of them supported by a different point of $Y\cap (L\cup R)$. Let $v'\subset \mathbb {P}^3$ be a general tangent vector
of $\mathbb {P}^3$ with $o'$ as its support (hence $\Res _H(v') =\{o'\}$). Let $v''\subset H$ be a general tangent vector of $H$ with $v''_{\red }= \{o''\}$.
Since $v'' \subset H$, we have $\Res _H(v'') =\emptyset$. Since $v''$ is general, it is not tangent to $T$ and hence $\Res _T(v'') =\{o''\}$.
Take $Y':= Y\cup T$, $v:= w\cup v'\cup v''$ and $S:= Y\cap (L\cup R) \cup \{o',o''\}$. We want to check that the quintuple $(Y',S,C,H,v)$ satisfies $R(m)$.
The scheme $v$ is a union of tangent vectors, one for each point of $S$. We have $\sharp (S) =(m+1)/2 +2 = (m+5)/2$. The set $S$ is contained in the smooth
conic $D$. It is sufficient to check that $h^i(\mathcal {I}_{mP\cup Y'\cup v}(m+2))=0$, $i=0,1$. Since $\Res _H(mP\cup Y'\cup v) = (m-1)P\cup Y\cup \{o'\}$,
$h^i(\mathcal {I}_{(m-1)P\cup Y\cup \{o\}}(m+1)) =0$, $i=0,1$, $o'\in D$, $(mP\cup Y'\cup v)\cap H = \{(m-1)P\cup w \cup T\cup
\{o'\}\cup v'' \cup ((Y\cap H)\setminus (Y\cap H)\cap (L\cup R))$ and $\Res _T(\{mP,H\}\cup T \cup v'' \cup (Y\cap H)\cup w) = \{mP,H\} \cup (Y\cap H)\cap w \cup \{o''\}$,
it is sufficient to prove that $h^i(H,\mathcal {I}_{\{mP,H\}\cup (Y\cap H) \cup w\cup \{o''\}}(m+1))=0$. We have $\deg (\{mP,H\}\cup (Y\cap H) \cup w\cup \{o''\})
= \binom{m+1}{2} +3m/2+3/2 +(m+1)/2 + 1 = \binom{m+3}{2}$. Use again that $h^1(H,\mathcal {I}_{\{mP,H\}\cup w}(m+1)) =0$ (as
in part (a) of the proof of Lemma \ref{a8}) and that $Y\cap H \setminus v_{\red}$ is general in $H$.
\end{proof}

\begin{lemma}\label{a10}
Fix an integer $d> a_{m,m+2}$ and let $X\subset \mathbb {P}^3$ be a general union of $d$ lines. Then $h^0(\mathcal {I}_{mP\cup X}(m+2)) =0$.
\end{lemma}

\begin{proof}
It is sufficient to prove the lemma when $d =a_{m,m+2}+1$.
First assume that $m$ is even. Take a solution $(Y,L,R,H)$ of $B(m)$. Since $h^0(\mathcal {I}_{mP\cup Y}(m+2)) =1$, we
have $h^0(\mathcal {I}_{mP\cup Y\cup D}(m+2)) =0$ for any line $D$ through a general point of $\mathbb {P}^3$.

Now assume that $m$ is odd. Let $W\subset \mathbb {P}^3$ be a general union of $a_{m-1,m+1}$ lines. Since $B(m-1)$ is true,
we have
$h^1(\mathcal {I}_{(m-1)P\cup W}(m+1)) =0$ and hence $h^0(\mathcal {I}_{(m-1)P\cup W \cup o}(m+1)) =0$ for a general $o\in \mathbb {P}^3$.
Let $M\subset \mathbb {P}^3$ be a general plane containing $\{P,o\}$. Let $L',R'\subset M$ be two general lines through $o$.
It is sufficient to prove that $h^0(\mathcal {I}_{mP\cup W\cup L'\cup R'\cup 2o}(m+2)) =0$.
Since $\Res _M(mP\cup W\cup L'\cup R'\cup 2o) =(m-1)P\cup W\cup \{o\}$, it is sufficient to prove that
$h^0(H,\mathcal {I}_{\{mP,H\}\cup (W\cap H) \cup L'\cup R'}(m+2)) =0$, i.e. $h^0(H,\mathcal {I}_{\{mP,H\}\cup (W\cap H)}(m)) =0$. Since $W\cap H$
is a general union of $a_{m-1,m+1} >m$ points of $H$, we have $h^0(H,\mathcal {I}_{\{mP,H\}\cup (W\cap H)}(m)) =0$. 
\end{proof}

Consider the following statement:

\quad {\bf Assertion} $H_{m,k}$, $m>0$, $k\ge m+2$: There exist a quintuple $(Y, Q, S,v,E)$ with the following properties:
\begin{enumerate}
\item $Y \in L(P,a_{m,k})$, $Q$ is a smooth quadric surface intersecting transversally $Y$, $P\notin Q$;
\item $S\subseteq Y\cap Q$, $\sharp (S) =b_{m,k}$ and $v\subset \mathbb {P}^3$ is a disjoint
union of tangent vectors with $v_{\red } =S$ and no connected component of $v$ contained in $Y$;
\item $E\subset Q$ is a disjoint union of $\lceil b_{m,k}/2\rceil$ lines, $S\subset E$ and each component of $E$ contains at
most two points;
\item $h^i(\mathcal {I}_{mP\cup Y\cup v}(k)) =0$, $i=0,1$.
\end{enumerate}

Take $(Y,Q,S,v,E)$ satisfying the first two conditions of the definition of $H_{m,k}$. We have $h^0(\mathcal {O}_{mP\cup Y\cup v}(k)) =\binom{k+3}{3}$
and hence $h^0(\mathcal {I}_{mP\cup Y\cup v}(k)) =h^1(\mathcal {I}_{mP\cup Y\cup v}(k))$. Now assume that $(Y,Q,S,v,E)$ satisfies the third condition
of the definition of $H_{m,k}$. If $b_{m,k}$ is even, then each line of $S$ contains exactly two points of $S$. If $b_{m,k}$ is odd, then $\sharp (S\cap L) =2$
for $(b_{m,k}-1)/2$ of the components of $E$, while $\sharp (S\cap L)=1$ for the other component.

From now on $Q\subset \mathbb {P}^3$ is a smooth quadric surface such that $P\notin Q$.

\begin{lemma}\label{a11}
$H_{m,m+3}$ is true for all $m >0$.
\end{lemma}

\begin{proof}
We have $a_{m,m+1}=m+2$, $b_{m,m+1}=0$, $a_{m,m+3} =2m+4$ and $b_{m,m+3} =4$ (Remark \ref{a1}). Let $Y\subset \mathbb {P}^3$ be a general union of $m+2$ lines. By \cite[Part (i)(c) of Theorem 4.2]{ccg1} we have $h^i(\mathcal {I}_Y(m+1)) =0$, $i=0,1$.
For a general $Y$ we may assume that $Y\cap Q$ is formed by $2m+4$ general points of $Q$. 
Let $F\subset Q$ be a general union of $m+2$ lines of type $(0,1)$. Fix $S_1\subset Y\cap Q$ such that $\sharp (S_1) =2$. Let
$E'\subset Q$ be the union of the lines of type $(1,0)$ containing a point of $S_1$. Fix $S_2\subset E'\cap F$ such that $\sharp (S_2\cap L)=1$ for
each component $L$ of $E'$ and that no component of $F$ contains two points of $S_2$. Set $S:=S_1\cup S_2$
and call $v\subset Q$ a general union of tangent vectors of $Q$ with $S$ as its support. We claim that
$h^i(\mathcal {I}_{mP\cup Y\cup F\cup v}(m+3)) =0$. Since $\Res _Q(mP\cup Y\cup F\cup v) =mP\cup Y$, $Q\cap (mP\cup Y\cup F\cup v)
= F\cup v \cup ((Y\cap Q)\setminus S_1)$, $\Res _F(F\cup v \cup ((Y\cap Q)\setminus S_1) = ((Y\cap Q)\setminus S_1)\cup v_{\red} = Y\cap Q  \cup S_2$
and $h^i(\mathcal {I}_{mP\cup Y}(m+1)))=0$, $i=0,1$, to prove the claim it is sufficient
to prove that $h^i(Q,\mathcal {I}_{((Y\cap Q)\setminus S_1)\cup v}(m+3,1)) =0$, $i=0,1$. We have $\deg (((Y\cap Q)\setminus S_1)\cup S)
2m+6$. Hence it is sufficient to use that $h^1(Q,\mathcal {I}_{S}(m+3,1)) =0$ ($S$ is the union of two degree 2 schemes on two different lines of type $(1,0)$)
and $(Y\cap Q)\setminus S_1$ is general in $Q$.
\end{proof}

\begin{lemma}\label{a12}
$H_{m,m+4}$ is true for all $m \ge 2$.
\end{lemma}

\begin{proof}
The proof depends on the parity of $m$.

\quad (a) First assume that $m$ is even. We have $a_{m,m+2} =3m/2+3$ and $b_{m,m+2}=1$.

\quad (a1) Assume for the moment $m\ge 6$ and hence $a_{m,m+4} = 5m/2+5$ and $b_{m,m+4} =10$ (Remark \ref{a1}). Let $Y\subset \mathbb {P}^3$ be a general union of $a_{m,m+2} =3m/2+3$ lines. Since $B(m)$ is
true (Lemma \ref{a8}) and $b_{m,m+2}=1$, we have $h^1(\mathcal {I}_{mP\cup Y}(m+2))=0$ and $h^0(\mathcal {I}_{mP \cup Y}(m+2)) =1$. 
The last equality implies $h^0(\mathcal {I}_{mP\cup Y}(m+1))=0$.

Let $T\subset \mathbb {P}^3$ be the only surface of degree $m+2$ containing $mP\cup Y$. Fix a system $x_0,x_1,x_2,x_3$ of homogeneous coordinates and 
let $f(x_0,x_1,x_2,x_3)$ be a degree $m+2$ homogeneous equation of $T$. In characteristic zero we have $\partial f/\partial x_i \ne 0$ for
at least one index $i$. Since $\partial f/\partial x_i \ne 0$ and $h^0(\mathcal {I}_{mP\cup Y}(m+1))=0$, we
have $\partial f/\partial x_i |mP\cup Y\ne 0$, i.e. $mP\cup Y\nsubseteq \Sing (T)$, i.e. $Y\nsubseteq \Sing (T)$. In characteristic $p>0$
we need to prove the existence of $Y\in L(P,3m/2+3)$ such that $h^0(\mathcal {I}_{mP\cup Y}(m+2))=1$ and at least one component of $Y$ is not contained in the singular
locus of the only degree $m$ hypersurface containing $mP\cup Y$. We get this using the proof that $B(m-1)$ implies $B(m)$ when $m$ is even (part (b)
of the proof of Lemma \ref{a8}). Take $(Y,L,R,S,O, H,v)$ satisfying $B(m-1)$ and use that $h^0(H,\mathcal {I}_{\{mP,H\}\cup 2L\cup 2R}(m+2)) =0$.

Since $Y\nsubseteq \Sing (T)$, there is $O\in Y$ and a tangent vector $w$ to $\mathbb {P}^3$ with $w\nsubseteq T$. Hence $h^0(\mathcal {I}_{mP\cup Y\cup w}(m+2))=0$
and so $h^1(\mathcal {I}_{mP\cup Y\cup w}(m+2))=0$.
Take as $Q$ a general quadric surface through $O$. We have $P\notin Q$ and $\Res _Q(mP\cup Y\cup w) =mP\cup Y\cup w$.
Moving the lines of $Y$ among the unions of $a_{m,m+2}$ disjoint lines of $\mathbb {P}^3$, one of them containing $O$, we may assume
that $(Y\cap Q)\setminus \{O\}$ is a general union of $3m+5$ points of $Q$. Let $F\subset Q$ be a union of $m+2$ distinct lines of type $(0,1)$ of
$Q$ with $\{O\} = Y\cap F$. Fix $S_1\subset (Y\cap Q)\setminus \{O\}$ with $\sharp (S_1) =5$. Let $E\subset Q$ be the union of the $5$ lines of type $(1,0)$
of $Q$ containing one point of $S_1$. Take $S_2\subset E\cap F$ such that each line of $E$ contains exactly one point of $S_2$ and each line of
$F$ contains at most one point of $S_2$. Set $S:= S_1\cup S_2$. Let $v\subset Q$ be a general union of tangent vectors of $Q$ with $v_{\red }=S$.
As in the proof of Lemma \ref{a11} it is sufficient to prove that $h^i(\mathcal {I}_{mP\cup Y\cup w\cup F\cup v}(m+4)) =0$,
$i=0,1$. Since $h^i(\mathcal {I}_{mP\cup Y\cup w}(m+2))=0$, $i=0,1$, it is sufficient to prove that $h^i(Q,\mathcal {I}_{(Y\cap Q)\cup F\cup v}(m+4))=0$,
i.e. $h^i(Q,\mathcal {I}_{(Y\cap Q)\cup v}(m+4,2)) =0$, $i=0,1$. We have $\deg ((Y\cap Q)\cup v) = 20 +(3m+6) -10$. Since $v$ is general, $\deg (v)=20$
and $S$ is general with the only restriction that $5$ lines of type $(1,0)$ of $Q$ contain each two points of $S$, we have
$h^1(\mathcal {I}_{v}(m+4,2)) =0$. Since $Y\cap Q\setminus S_1$ is general in $Q$, we get  $h^i(Q,\mathcal {I}_{(Y\cap Q)\cup v}(m+4,2)) =0$, $i=0,1$.

\quad (a2) Now assume $m \in \{2,4\}$. We have $a_{m,m+4} =5m/2+6$ and $b_{m,m+4} = 5-m$. Now $F$ is a union of $m+3$ lines, $\sharp (S_1) =2-m/2$,
$\deg (E) =3-m/2$
and $\sharp (S_2) = 3-m/2$.

\quad (b) Now assume that $m$ is odd.

\quad (b1) Assume $m\ge 17$. We have $a_{m,m+2} =3m/2+5/2$ and $b_{m,m+2} = (m+5)/2$. Since $m\ge 17$, we have $a_{m,m+4} =5m/2+9/2$ and $b_{m,m+4} = (m+25)/2$
(Remark \ref{a1}).
Take a solution $(Y,S,D,H,v)$ of $R(m)$ (Lemma \ref{a12.1}). Let $Q\subset \mathbb {P}^3$ be a general quadric surface containing $D$. Since $P\notin D$, then
$P\notin Q$. The quadric $Q$ is smooth and it intersects transversally $Q$. By the semicontinuity theorem for cohomology (\cite[III.12.8]{h}) for general $v$ we may also assume that no connected component of
$v$ is contained in $Q$, i.e. that $(mP\cup Y\cup v) = Y\cap Q$ (as schemes) and that $\Res _Q(mP\cup Y\cup v) = mP\cup Y\cup v$. We may move each component of $Y$ keeping fixed its point in $D$. Hence we may assume that $Y\cap (Q\setminus D)$ is a general subset of $Q$ with cardinality $2a_{m,m+2} -b_{m,m+2}$.
We have $a_{m,m+4}-a_{m,m+2} = m+2 \ge (m+5)/2 =b_{m,m+2}$. Let $F\subset Q$ be a disjoint union of $m+2$ lines of type $(0,1)$
with the only restriction that $F\cap Y = Y\cap D$ (it exists, because $m+2\ge b_{m,m+2}$).
Fix $S_1\subseteq (Y\cap Q)\setminus S$ such that $\sharp (S_1) =\lfloor b_{m,m+4}/2\rfloor$ (it exists because $2a_{m,m+2} = 3m+5\ge (m+5)/2 + m \ge b_{m,m+2} +\lfloor 
b_{m,m+4}/2\rfloor$). Let $E_1$ be the union of the lines of type $(1,0)$
of $Q$ containing one point of $S_1$. If $b_{m,m+4}$ is even, then set $E':= E$. If $b_{m,m+4}$ is odd, then let $E'$ be the union of $E_1$ and a general line of type $(1,0)$
of $Q$. Let $S_2\subset E'\cap E''$ be the union of one point for each component of $E'$, with the restriction
that $S_2\cap S_1=\emptyset$ and that each point of $S_2$ is contained in a different line of $E''$; we may find such a set $S_2$, because $E''\cap S_1=\emptyset$
and $\deg (E'') = a_{m,m+4}-a_{m,m+2}  =m+2\ge \lceil b_{m,m+4}/2\rceil$. Let $v'\subset Q$ be a general union of $b_{m,m+4}$ tangent vectors of $Q$ with $v'_{\red } =S'$.
Since $v'$ is general, no connected component of $v'$ is contained in $E''$ (hence $\Res _{E''}((Y\cap Q) \cup v') = Y\cap (Q\setminus E'')  \sqcup S' = ((Y\cap Q)\setminus S)\sqcup S'$.

\quad (b2) Assume $m \in \{3,5,7,9,11,13,15\}$. We have $a_{m,m+4} -a_{m,m+2} = m+3$ and $b_{m,m+4} = (15-m)/2$. We make the construction
of step (b1) with $\deg (F) = m+3$, $\sharp (S_1) = \lfloor (15-m)/4\rfloor$, and $\deg (E) =\sharp (S_2) =\lceil (15-m)/4\rceil$.
\end{proof}

\begin{lemma}\label{a12.0}
For all integers $k \ge m+3$ we have $a_{m,k}-a_{m,k-2} \ge a_{0,k}-a_{0,k-2} -1\ge \lceil k/2\rceil$.
\end{lemma}

\begin{proof}
From (\ref{eqa2}) and the same equation for $m= 0$ we get
\begin{align*}
&2a_{m,k-2} +(k+1)(a_{m,k}-a_{m,k-2}) +b_{m,k}-b_{m-k-2} = \\
&2a_{0,k} +(k+1)(a_{0,k}-a_{0,k-2}) + b_{0,k}-b_{0,k-2}
\end{align*}
We have $b_{0,x} =0$ if $x\equiv 0,1\pmod{3}$ and $b_{0,x} =(x+1)/3$ if $x\equiv 2\pmod{3}$. The definitions
of the integers $a_{m,k-2}$ and $a_{0,k-2}$ give $a_{m,k-2} \le a_{0,k-2}$, proving the lemma.
\end{proof}

\begin{lemma}\label{a13}
$H_{m,k}$ is true for all $k\ge m+3$.
\end{lemma}

\begin{proof}
By Lemmas \ref{a11} and \ref{a12} we may assume $k\ge m+5$ and that $H_{m,k-2}$ is true.
Fix a solution $(Y,Q,S,v,E)$ of $H_{m,k-2}$. Deforming if necessary each line of $Y$ we may assume
that $(Q\cap Y)\setminus S$ is a general subset of $Q$. Taking instead of $v$ a union of general tangent vectors of $\mathbb {P}^3$ with the points of $S$ as their support
we may assume that no connected component of $v$ is contained in $Q$. Therefore $\Res _Q(Y\cup v) = Y\cup v$ and $(Y\cup v)\cap Q=Y\cap Q$ (as schemes).
Call $(0,1)$ the ruling of $Q$ containing $E$ (any ruling of $Q$ if $b_{m,k-2} =0$ and hence $E=\emptyset$).
Lemma \ref{a12.0} gives $a_{m,k}-a_{m,k-2} \ge \deg (E)$. Let $F\subset Q$ be a general union of $a_{m,k}-a_{m,k-2} -\lceil b_{m,k-2}/2\rceil$ lines of type $(0,1)$
of $Q$. Set $E'':= E\cup F$. 

\quad {\emph {Claim 1:} If $k\ge m+5$, then $2a_{m,k-2} \ge k-2 +k/2$.

\quad {\emph {Proof of Claim 1:}} We have $2(k-1)a_{m,k-2} +2b_{m,k-2} =2\binom{k+3}{3} -2\binom{m+2}{3}$ and $b_{m,k-2} \le k-2$. Set $\psi (k,m):= 2\binom{k+3}{3}
-2\binom{m+2}{3} -(k-1)(k-2+k/2) -2k+4$. It is sufficient to prove that $\psi (k,m)\ge 0$ for all $k\ge m+5$. We have
$\psi (m+5,m) = (m+8)(m+7)(m+6)/3 -(m+2)(m+1)m/3 -(m+4)(3m+11)/2 -2m+6\ge 0$ and $\psi (k+1,m)\ge \psi (k,m)$ for all $k\ge m+5$.

Fix $S_1\subseteq (Y\cap Q)\setminus S$ such that $\sharp (S_1) =\lfloor b_{m,k}\rfloor$ (it exists by Claim 1 and the inequalities
$b_{m,k-2} \le k-2$, $b_{m,k} \le k$). Let $E_1\subset Q$ be the union of the lines of type $(1,0)$
of $Q$ containing one point of $S_1$. If $b_{m,k}$ is even, then set $E':= E$. If $b_{m,k}$ is odd, then let $E'$ be the union of $E_1$ and a general line of type $(1,0)$
of $Q$. Let $S_2\subset E'\cap E''$ be the union of one point for each component of $E'$, with the restriction
that $S_2\cap S_1=\emptyset$ and that each point of $S_2$ is contained in a different line of $E''$; we may find such a set $S_2$, because $E''\cap S_1=\emptyset$
and $\deg (E'') = a_{m,k}-a_{m,k-2} \ge \lceil b_{m,k}/2\rceil$ (Lemma \ref{a12.0}). Set $S':= S_1\cup S_2$. Let $v'\subset Q$ be a general union of $b_{m,k}$ tangent vectors of $Q$ with $v'_{\red } =S'$.
Since $v'$ is general, no connected component of $v'$ is contained in $E''$ (hence $\Res _{E''}((Y\cap Q) \cup v') = Y\cap (Q\setminus E'')  \sqcup S' = ((Y\cap Q)\setminus S)\sqcup S'$). 

\quad {\emph {Claim 2:}}
We claim that
$h^i(\mathcal {I}_{mP\cup Y\cup v\cup E''\cup v'}(k)) =0$, $i=0,1$. 

\quad {\emph {Proof of Claim 2:}} Since $\Res_Q(mP\cup Y\cup v\cup E''\cup v') =mP\cup Y\cup v$, $Q\cap (mP\cup Y\cup v\cup E'\cup v')
= (Y\cap Q)\cup E'' \cup v'$
and $h^i(\mathcal {I}_{mP\cup Y\cup v}(k-2)) =0$, $i=0,1$, it is sufficient to prove that $h^i(Q,\mathcal {I}_{(Y\cup Q)\cup v'\cup E'}(k)) =0$, $i=0,1$,
i.e. $h^i(Q,\mathcal {I}_{((Y\cap Q)\setminus S)\cup S'}(k,k-a_{m,k}+a_{m,k-2})) =0$, $i=0,1$. By (\ref{eqa2}) we have
$\sharp ((Y\cap Q)\setminus S)\cup S') = (k+1)(k+1-a_{m,k}+a_{m,k-2})$. Hence it is sufficient to prove that
the set $((Y\cap Q)\setminus S)\cup S'$ gives independent conditions to the linear system $|\mathcal {O}_Q(k,k-a_{m,k}+a_{m,k-2})|$. Since
$(Y\cap Q)\setminus S$ is general in $Q$, it is sufficient to prove that $S'$ gives independent conditions to $|\mathcal {O}_Q(k,k-a_{m,k}+a_{m,k-2})|$. This is true
since $S_1$ is general and hence the only restriction on the subset $S'$ of $Q$ is that each line of $E_1$ contains two points of $S'$.

}We may deform $Y\cup v\cup E''$ to a family of members of $L(P,a_{m,k})$ containing the points of $S'$ and
whose general member, $Y'$, intersects transversally $Q$, because each line of $E''$ contains at most one point of $Q$.
The quintuple $(Y',Q,S',v',E')$ satisfies $H_{m,k}$.\end{proof}

\begin{proof}[Proof of Theorem \ref{i1}:] Fix positive integers $m, d$ with critical value $k$. Hence $a_{m,k-1} < d \le a_{m,k}$. See Remark \ref{a8.0} and Lemma \ref{a9} for the
cases $k=m,m+1,m+2$. Hence we may assume $k\ge m+3$ and that the theorem is true for the integers $d$ such that $(m,d)$ has critical value $<k$.
Since $L(P,d)$ is irreducible, it is sufficient to prove the existence of $A, B\in L(P,d)$ such
that $h^0(\mathcal {I}_{mP\cup B}(k-1)) =0$ and $h^1(\mathcal {I}_{mP\cup A}(k))=0$. Let $Q\subset \mathbb {P}^3$ be a smooth quadric surface such that $P\notin Q$.

\quad (a) In this step we prove the existence of $A$. Since any element of $L(P,d)$ is a union of some of the connected components of an element of $L(P,a_{m,k})$,
it is sufficient to do the case $d=a_{m,k}$. Fix a solution $(Y,Q,S,v,E)$ of $H_{m,k-2}$. Deforming if necessary each line of $Y$ we may assume
that $(Q\cap Y)\setminus S$ is a general subset of $Q$. Taking instead of $v$ a union of general tangent vectors of $\mathbb {P}^3$ with the points of $S$ as their support
we may assume that no connected component of $v$ is contained in $Q$. Therefore $\Res _Q(Y\cup v) = Y\cup v$ and $(Y\cup v)\cap Q=Y\cap Q$ (as schemes).
Call $(0,1)$ the ruling of $Q$ containing $E$ (any ruling of $Q$ if $b_{m,k-2} =0$ and hence $E=\emptyset$).
Lemma \ref{a12.0} gives $a_{m,k}-a_{m,k-2} \ge \deg (E)$. Let $F\subset Q$ be a general union of $a_{m,k}-a_{m,k-2} -\lceil b_{m,k-2}/2\rceil$ lines of type $(0,1)$
of $Q$. The scheme $Y\cup E\cup F\cup v$ is a flat limit of a family of disjoint unions of $a_{m,k}$ lines, none of them containing $P$.
By the semicontinuity theorem for cohomology to prove the existence of $A$ it is sufficient to prove that $h^1(\mathcal {I}_{mP\cup Y \cup E\cup F\cup v}(k)) =0$.
We proved a more difficult vanishing in the proof of Lemma \ref{a13} (copy it without $v'$).

\quad (b) In this step we prove the existence of $B$. By Lemma \ref{a10} we may assume $k-1\ge m+3$. Since $b>a_{m,k-1}$, to prove the existence of $B$ it is sufficient to prove it when $d = a_{m,k-1}+1$. 

\quad (b1) Assume for the moment $k-1\ge m+4$, i.e. $k-3\ge m+2$.
Fix a solution $(Y,Q,S,v,E)$ of $H_{m,k-3}$. Deforming if necessary each line of $Y$ we may assume
that $(Q\cap Y)\setminus S$ is a general subset of $Q$. Taking instead of $v$ a union of general tangent vectors of $\mathbb {P}^3$ with the points of $S$ as their support
we may assume that no connected component of $v$ is contained in $Q$. Therefore $\Res _Q(Y\cup v) = Y\cup v$ and $(Y\cup v)\cap Q=Y\cap Q$ (as schemes).
Call $(0,1)$ the ruling of $Q$ containing $E$ (any ruling of $Q$ if $b_{m,k-3} =0$ and hence $E=\emptyset$).
Lemma \ref{a12.0} gives $a_{m,k-1}-a_{m,k-3} \ge \deg (E)$. Let $F\subset Q$ be a general union of $a_{m,k}-a_{m,k-3} -\lceil b_{m,k-2}/2\rceil +1$ lines of type $(0,1)$
of $Q$. The scheme $Y\cup E\cup F\cup v$ is a flat limit of a family of disjoint unions of $a_{m,k-1}+1$ lines, none of them containing $P$.
By the semicontinuity theorem for cohomology to prove the existence of $B$ it is sufficient to prove that $h^0(\mathcal {I}_{mP\cup Y \cup E\cup F\cup v}(k-1)) =0$.
Since $\Res _Q(mP\cup Y\cup F\cup E\cup F\cup v) = mP\cup Y\cup E\cup v$ and $Q\cap (mP\cup Y\cup F\cup E\cup F\cup v) = (Y\cap Q) \cup E\cup F$,
it is sufficient to prove that $h^0(Q,\mathcal {I}_{F\cup E \cup (Y\cap Q)}(k-1)) = 0$, i.e. that $h^0(Q,\mathcal {I}_{(Y\cap Q)\setminus S}(k-1,k-a_{m,k-1}+a_{m,k-3}-2)) =0$.
Since $(Y\cap Q)\setminus S$ is general in $Q$, it is sufficient to prove that $\sharp (Y\cap Q) -\sharp (S) \ge k(k-a_{m,k-1}+a_{m,k-3}-1)$. By
(\ref{eqa2}) for the integer $k':= k-1$ we have $\sharp (Y\cap Q) -\sharp (S) = k(k-a_{m,k-1}+a_{m,k-3}-1) +k-b_{m,k-1} > k(k-a_{m,k-1}+a_{m,k-3}-1)$.

\quad (b2) Now assume $k=m+4$. We modify the proof of $H_{m,m+3}$ (Lemma \ref{a11}). We have $a_{m,m+1}=m+2$, $b_{m,m+1}=0$, $a_{m,m+3} =2m+4$ and $b_{m,m+3} =4$ (Remark \ref{a1}). Let $Y\subset \mathbb {P}^3$ be a general union of $m+2$ lines. By \cite[Part (i)(c) of Theorem 4.2]{ccg1} we have $h^i(\mathcal {I}_Y(m+1)) =0$, $i=0,1$.
For a general $Y$ we may assume that $Y\cap Q$ is formed by $2m+4$ general points of $Q$. 
Let $F\subset Q$ be a general union of $m+3$ lines of type $(0,1)$. Use $Y\cup F$. Since $Y\cap Q$ is a general subset of $Q$ with cardinality
$2m+4$, we have $h^0(Q,\mathcal {I}_{Q\cap Y}(m+3,0)) =0$. We also have $Y\cap F=\emptyset$ and hence $h^0(Q,\mathcal {I}_{(Y\cap Q)\cup F}(m+3))=0$. 
\end{proof}

\providecommand{\bysame}{\leavevmode\hbox to3em{\hrulefill}\thinspace}

\end{document}